\documentclass[11pt, a4paper, twoside]{amsart}
\usepackage{amsfonts}
\usepackage{mathrsfs}
\usepackage{amsmath}
\usepackage{amssymb}
\usepackage{fancyhdr}
\usepackage{graphicx}

\allowdisplaybreaks

\renewcommand{\thefootnote}{\arabic {footnote}}

\newcommand\CC{\mathbb{C}}

\newcommand\HH{\mathbb{H}}

\renewcommand\Re{\operatorname{Re}}
\renewcommand\Im{\operatorname{Im}}

\newtheorem{theorem}{Theorem}
\newtheorem{lemma}[theorem]{Lemma}
\newtheorem{proposition}[theorem]{Proposition}
\newtheorem{corollary}[theorem]{Corollary}

\theoremstyle{remark}

\newcommand\blfootnote[1]{%
  \begingroup
  \renewcommand\thefootnote{}\footnote{#1}%
  \addtocounter{footnote}{-1}%
  \endgroup
}

\begin{document}
\title{\textbf{
The Sobolev Inequalities on Real Hyperbolic Spaces and  Eigenvalue Bounds for Schr\"odinger Operators with Complex Potentials}}
\author{Xi Chen }

\begin{abstract}In this paper, we prove the uniform estimates for the resolvent $(\Delta - \alpha)^{-1}$ as a map from $L^q$ to $L^{q'}$ on real hyperbolic space $\mathbb{H}^n$ where $\alpha \in \mathbb{C}\setminus [(n - 1)^2/4, \infty)$ and $2n/(n + 2) \leq q < 2$. In contrast with analogous results on Euclidean space $\mathbb{R}^n$, the exponent $q$ here can be arbitrarily close to $2$. This striking improvement is due to two non-Euclidean features of hyperbolic space: the Kunze-Stein phenomenon and the exponential decay of the spectral measure. In addition, we apply this result to the study of eigenvalue bounds of the Schr\"{o}dinger operator with a complex potential. The improved Sobolev inequality results in a better long range eigenvalue bound on $\mathbb{H}^n$ than that on $\mathbb{R}^n$.\end{abstract}

\maketitle
\section{Introduction}

\blfootnote{\textup{2010} \textit{Mathematics Subject Classification}: \textup{58C40},  \textup{35J10},  \textup{35P15}}

Let $\mathbb{H}^{n}$ be the $n$-dimensional real hyperbolic space for $n > 2$ and $\Delta$ the Laplacian on $\mathbb{H}^{n}$. We are concerned with the $(L^s, L^r)$ type estimates for the resolvent $(\Delta - \alpha)^{-1}$ with $\alpha \in \mathbb{C}$, i.e. the norm estimates for $$\|(\Delta - \alpha)^{-1}\|_{L^s(\mathbb{H}^{n}) \rightarrow L^r(\mathbb{H}^{n})}.$$

This type of resolvent estimates traces back to the classical Sobolev inequality on Euclidean space $\mathbb{R}^n$,
$$\|\Delta_{\mathbb{R}^n}^{-1}\|_{L^{2n/(n+2)}(\mathbb{R}^n)\to L^{2n/(n-2)}(\mathbb{R}^n)} < C.$$
One can regard this inequality as the $L^{2n/(n+2)}-L^{2n/(n-2)}$ boundedness of the resolvent at $\alpha=0$. More generally, Kenig-Ruiz-Sogge \cite{Kenig-Ruiz-Sogge} extended this to the non-zero energies $\alpha\neq 0$ and proved that for $\alpha\in \CC\setminus [0,\infty)$ and $ 1 < s, r < \infty$,
$$ \quad \mbox{with} \quad \frac{1}{s} = \frac{1}{r} + \frac{2}{n},\quad \min\bigg\{|\frac{1}{r} - \frac{1}{2}|, |\frac{1}{s} - \frac{1}{2}|\bigg\} > \frac{1}{2n},$$ the following uniform Sobolev inequality holds,
\begin{eqnarray}\label{eqn:Kenig-Ruiz-Sogge}\|(\Delta_{\mathbb{R}^n} - \alpha)^{-1}\|_{L^s(\mathbb{R}^n) \rightarrow L^r(\mathbb{R}^n)} < C.\end{eqnarray}

The inequalities of $(L^q, L^{q'})$ type are of particular interests, where $1/q + 1/q' = 1$. As is well-known, they are closely tied to a number of applications in PDEs. Let us mention the eigenvalue bounds for Schr\"odinger operators with complex potentials, the endpoint Strichartz estimates for Schr\"{o}dinger equations, the Carleman inequalities to deduce relevant unique continuation theorem.

 Guillarmou-Hassell \cite{Guillarmou-Hassell} proved that, \begin{equation*}\|(\Delta_M - \alpha)^{-1}\|_{L^q(M)\rightarrow L^{q'}(M)} \leq C |\alpha|^{n(1/q - 1/2) - 1},\end{equation*} for $2n/(n + 2) \leq q \leq 2(n + 1)/(n + 3)$ on $M$, a non-trapping $n$-dimensional asymptotically Euclidean space. We remark that the range of the exponent $q$ here is  linked with the range of the exponent in the Stein-Tomas restriction estimates or essentially the following spectral measure estimates, studied in \cite{Guillarmou-Hassell-Sikora}, \begin{equation}\label{eqn : spectral measure on Euclidean}\Big|  \bigg( \big( \frac{d}{d\lambda} \big)^j dE_{\sqrt{\Delta_M}}(\lambda) \bigg) (z,z')  \Big| \leq
C \lambda^{n-j} (1 + \rho \lambda)^{-n/2+j},\end{equation} where $\rho$ is the geodesic distance function on $M$ and the spectral measure $dE_{\sqrt{\Delta_M}}(\lambda)$ is defined through spectral theorem by \[f(\sqrt{\Delta_M}) = \int_0^\infty f(\lambda) dE_{\sqrt{\Delta_M}}(\lambda)\, d\lambda,\] for all bounded functions $f$.
 It is also worth pointing out that Knapp's counterexample shows that $2(n + 1)/(n + 3)$ is the upper bound of the exponents for Stein-Tomas estimates on asymptotically Euclidean spaces. Therefore, one can not obtain uniform Sobolev estimates on $\mathbb{R}^n$ for $p$ close to $2$.

In this paper, we shall study the uniform Sobolev inequalities of $(L^q, L^{q'})$ type on real hyperbolic spaces. More precisely, we prove \begin{theorem}\label{thm : Sobolev on hyperbolic}Suppose $L = \Delta - (n - 1)^2/4$ and $\alpha \in \mathbb{C} \setminus [0, \infty)$.

On the one hand, for high energies $|\alpha| > 1$,  we have that  \begin{equation}\label{eqn : q-q' sobolev large spectral parameter}\|(L - \alpha)^{-1}\|_{L^q(\mathbb{H}^{n})\rightarrow L^{q'}(\mathbb{H}^{n})} \leq C |\alpha|^{1/2 - 1/q}, \end{equation} if $2(n + 1)/(n + 3) \leq q < 2$;\begin{equation}\label{eqn : q-q' sobolev large spectral parameter (normal range)}\|(L - \alpha)^{-1}\|_{L^q(\mathbb{H}^{n})\rightarrow L^{q'}(\mathbb{H}^{n})} \leq C |\alpha|^{n(1/q - 1/2) - 1} ,\end{equation} if $2n/(n + 2) \leq q \leq 2(n + 1)/(n + 3)$.

On the other hand, for low energies $|\alpha| < 1$,  we have that for $ 2n/(n+2) \leq q < 2$,
\begin{equation}\label{eqn : q-q' sobolev small spectral parameter} \|(L - \alpha)^{-1}\|_{L^{q}(\mathbb{H}^{n}) \rightarrow L^{q'}(\mathbb{H}^{n})} \leq C . \end{equation}\end{theorem}

We remark that the constant $C$ here is $q$-dependent and it blows up as $q$ goes to $2$.
Because all $q \in [2n/(n + 2), 2)$ satisfy $1/2 - 1/q < 0$ and $n(1/q - 1/2) - 1 \leq 0$, Theorem \ref{thm : Sobolev on hyperbolic} trivially implies

\begin{corollary}For all $\alpha \in \mathbb{C} \setminus [0, \infty)$,
\begin{equation}\label{eqn : q-q' sobolev all spectral parameters}\|(L - \alpha)^{-1}\|_{L^q(\mathbb{H}^{n})\rightarrow L^{q'}(\mathbb{H}^{n})} \leq C |\alpha|^{1/2 - 1/q}, \end{equation} if $2(n + 1)/(n + 3) \leq q < 2$;\begin{equation}\label{eqn : q-q' sobolev all spectral parameters (normal range)}\|(L - \alpha)^{-1}\|_{L^q(\mathbb{H}^{n})\rightarrow L^{q'}(\mathbb{H}^{n})} \leq C |\alpha|^{n(1/q - 1/2) - 1} ,\end{equation} if $2n/(n + 2) \leq q \leq 2(n + 1)/(n + 3)$.
\end{corollary}

In contrast with the results on (asymptotically) Euclidean space, the exponent $q$ in the results on $\mathbb{H}^n$ can be arbitrarily close to $2$. We can interpret this striking phenomenon as the consequences of two non-Euclidean features of $\mathbb{H}^n$: the Kunze-Stein phenomenon and the exponential decay of the spectral measure.

The Kunze-Stein phenomenon, proved by Cowling \cite{Cowling} on connected semi-simple Lie groups $\mathbb{G}$ with finite center, says that the convolution inequality, $$L^2(\mathbb{G}) \ast L^p(\mathbb{G}) \subset L^2(\mathbb{G}),$$ holds for $p \in [1, 2)$.  Note that above inequalities on Euclidean space, by Young's equality, are only valid for $p = 1$. If $\mathbb{G}$ is in addition of real rank $1$, it can be refined in terms of Lorentz spaces. Cowling-Meda-Setti \cite{Cowling-Meda-Setti} and Ionescu \cite{Ionescu} showed $$L^{p,a}(\mathbb{G}) \ast L^{p, b}(\mathbb{G}) \subset L^{p, c}(\mathbb{G}), $$ provided
$$\left\{ \begin{array}{ll}
1/a + 1/b \geq 1 + 1/c, &\mbox{if $p \in (1, 2)$}\\
a = 1, b = 1, c = \infty, &\mbox{if $p = 2$}
\end{array} \right. .$$ A useful corollary of this on $\mathbb{H}^n$, proved by Anker-Pierfelice-Vallarino \cite{Anker-Pierfelice-Vallarino-CPDE-2011}, is that for a radial kernel $\kappa(\rho)$
\begin{equation}
\| f * \kappa \|_{L^{q'}(\HH^{n})} \leq C_q \| f \|_{L^{q}(\HH^{n})} \bigg( \int_0^\infty (\sinh \rho)^{n-1} (1+\rho) e^{-(n-1)\rho/2} |\kappa(\rho)|^{q'/2} \, d\rho \bigg)^{2/q'},
\label{KS}\end{equation} with $1 < q \leq 2$.  Applying \eqref{KS} to the resolvent kernel $e^{\pm \imath \rho \lambda}/\sinh(\rho)$, with sufficiently small $|\Im \lambda|$, on $\mathbb{H}^3$  immediately proves the $(L^q, L^{q'})$ Sobolev inequalities for any $6/5 < q < 2$.

The proof, by Kenig-Ruiz-Sogge, of uniform resolvent estimates is closely tied to the Stein-Tomas restriction theorem \cite{Beijing lecture, Tomas}, which says for any function $f \in L^p(\mathbb{R}^n)$ with $1 \leq p \leq 2(n + 1)/(n + 3)$ one can meaningfully restrict its Fourier transform $\hat{f}$ to the sphere $\mathbb{S}^{n-1}$. By a $TT^\ast$ argument, one can rewrite the restriction theorem, in terms of the spectral measure, as $$\|dE_{\sqrt{\Delta_{\mathbb{R}^{n}}}}(\lambda)\|_{L^p(\mathbb{R}^n) \rightarrow L^{p'}(\mathbb{R}^n)} \leq C \lambda^{n(1/p - 1/p') - 1},$$ provided $1 \leq p \leq 2(n + 1)/(n + 3)$. This inequality was proved by Guillarmou-Hassell-Sikora \cite{Guillarmou-Hassell-Sikora} (see also \cite{Chen-MathZ} for a complete proof) on asymptotically Euclidean spaces via the pointwise estimates \eqref{eqn : spectral measure on Euclidean} for the spectral measure.

Interestingly, all of these estimates at high energies on real hyperbolic spaces turn out to be better than on Euclidean spaces. Hassell and the author \cite{Chen-Hassell2} proved
\begin{theorem}\label{thm:kernelbounds} Let $\rho$ be the distance function between $z, z' \in \mathbb{H}^{n}$.
For  $\lambda \geq 1$, the spectral measure of $\sqrt{L}$ satisfies the following pointwise estimates,
\begin{equation}\label{eqn : spectral measure upper bound}\Big|  \bigg( \big( \frac{d}{d\lambda} \big)^j dE_{\sqrt{L}}(\lambda) \bigg) (z,z')  \Big| \leq
\left\{\begin{array}{ll}
C \lambda^{n - 1 -j} (1 + \rho \lambda)^{-(n - 1)/2+j},& \text{ for } \rho \leq 1 \\
C \lambda^{(n - 1)/2} \rho^{j} e^{-(n - 1)\rho/2},& \text{ for } \rho
\geq 1.\end{array}\right.
\end{equation}
Moreover, one has the Stein-Tomas estimates at high energies, for $1 \leq p < 2$, \begin{equation}\label{eqn:high energy restriction}
\|dE_{\sqrt{L}} (\lambda)\|_{L^p(\HH^{n}) \rightarrow L^{p^\prime}(\HH^{n})} \leq
\left\{\begin{array}{ll}
C \lambda^{n(1/p - 1/p') - 1}, & 1 \leq p \leq \frac{2(n+1)}{n+3}, \\
C \lambda^{(n - 1)(1/p - 1/2)}, & \frac{2(n+1)}{n+3} \leq p < 2.
\end{array}\right.
\end{equation}
\end{theorem}
Here the exponential decay in space of the spectral measure was exploited to prove the broader Stein-Tomas estimates for $\frac{2(n+1)}{n+3} \leq p < 2$. Therefore it is natural to expect such better spectral measure estimates will lead to a larger set of exponents for uniform Sobolev inequalities.

Resolvent estimates on real hyperbolic spaces have been studied before. For exponents away from $2$, Huang-Sogge \cite{Huang-Sogge} generalized \eqref{eqn:Kenig-Ruiz-Sogge} at high energies to $\mathbb{H}^{n + 1}$. For $p = 2$, Melrose-S\'{a} Barreto-Vasy \cite{Melrose-Sa Barreto-Vasy} established weighted $L^2$ estimates. The weight (or a compact cut-off) is essential for $p = 2$, since $\lim_{p \rightarrow 2 - 0}\|R(\lambda)\|_{L^p \rightarrow L^{p'}}$ blows up. However, we believe that Theorem \ref{thm : Sobolev on hyperbolic} is the first result of uniform Sobolev inequalities with non-Euclidean features of real hyperbolic spaces.

Moreover, one should be able to generalize our results to general non-trapping asymptotically hyperbolic manifolds (even if there are pairs of conjugate points) by employing the diagonal estimates for the spectral measure in \cite{Chen-Hassell2} as well as adapting the off-diagonal arguments in \cite{Guillarmou-Hassell}.

As was mentioned, the Sobolev inequalities can be applied to tackling a number of problems. In the present paper, we are concerned about the eigenvalue bounds for Schr\"odinger operators. Specifically, given a complex-valued potential $V \in L^p(\mathbb{H}^n)$ with $p \geq n/2$, what can we say about the spectrum of the Schr\"odinger operator $\Delta + V$? Adapting previous work for abstract operators on Hilbert spaces, due to Frank \cite{Frank-Simon III}, one can easily prove that the spectrum of $\Delta + V$ consists of the essential spectrum of $\Delta$ and isolated eigenvalues of finite algebraic multiplicity. The essential spectrum of $\Delta$ is the half real line $((n - 1)^2/4, \infty) \subset \mathbb{R}$. Nonetheless, what do we know about those isolated eigenvalues from the potential $V$?

This problem has been intensively studied on Euclidean space $\mathbb{R}^n$, including  the location of individual eigenvalues and the distribution of eigenvalues. Here we are interested in the location of individual eigenvalues, that is to prove the eigenvalues $\lambda$ are located in a ball, the size of which is controlled by the norm of the potential. In $\mathbb{R}^1$, Abramov-Aslanyan-Davies \cite{Abramov-Aslanyan-Davies} proved that
$$|\lambda|^{1/2} \leq C \int_{\mathbb{R}} |V|.$$ In higher dimensions $\mathbb{R}^n$, we write $p = \gamma + n/2$ in connection with Lieb-Thirring inequalities. Frank \cite{Frank-BLMS}, Frank-Sabin \cite{Frank-Sabin}, Frank-Simon \cite{Frank-Simon II} studied the short range case $0 < \gamma \leq 1/2$ and showed that $$|\lambda|^\gamma \leq C_{\gamma, n} \int_{\mathbb{R}^n} |V|^{\gamma + n/2}, \quad \mbox{for $n \geq 2$}.$$ On the other hand, Frank \cite{Frank-Simon III} established the following long range result for $\gamma > 1/2$, $$d(\lambda)^{\gamma - 1/2} |\lambda|^{1/2} \leq C_{\gamma, n} \int_{\mathbb{R}^n} |V|^{\gamma + n/2},$$ where $d(\lambda) = \text{dist}\,(\lambda, [0, \infty))$. Furthermore, Guillarmou-Hassell-Krupchyk \cite{Guillarmou-Hassell-Krupchyk} generalized all of these results to a large class of non-trapping asymptotically Euclidean manifolds for $n \geq 3$. We also refer the reader to \cite{Frank-Simon III, Guillarmou-Hassell-Krupchyk} for a complete list of references. Furthermore, this also has been considered on $\mathbb{H}^2$. Hansmann \cite{Hansmann} recently generalized these classical results to the hyperbolic plane.

In the present paper, we consider analogous eigenvalue bounds on $\mathbb{H}^n$. On the one hand, we prove similar short range results, aligned with the Sobolev inequality \eqref{eqn : q-q' sobolev all spectral parameters (normal range)},  \begin{theorem}\label{thm : short range}For $0 < \gamma \leq 1/2$ and $n \geq 3$, we have $$|\lambda|^\gamma \leq C_{\gamma, n} \int_{\mathbb{H}^n} |V|^{\gamma + n/2}.$$ \end{theorem} On the other hand, we, by means of the Sobolev inequality \eqref{eqn : q-q' sobolev all spectral parameters}, obtain the following better long range results  \begin{theorem}\label{thm : long range}For $\gamma \geq 1/2$ and $n \geq 3$, we have $$|\lambda|^{1/2} \leq C_{\gamma, n} \int_{\mathbb{H}^n} |V|^{\gamma + n/2}.$$\end{theorem}
In addition, eigenvalues do not exist if the potential is too 'small'.\begin{theorem}\label{thm : no eigenvalue}
 If $\|V\|_{L^{\gamma + n/2}(\mathbb{H}^n)}$ with $\gamma \geq 0$ is sufficiently small, the Schr\"odinger operator $\Delta + V$ has no eigenvalues.
\end{theorem}

The significant improvement in the long range case also results from the better Sobolev estimates on hyperbolic space. Due to lack of uniform Sobolev estimates for $2(n + 1)/(n + 3) < p < 2$ on $\mathbb{R}^n$, one has to interpolate the Stein-Tomas estimates on $\mathbb{R}^n$ with the crude $L^2$ estimates with the factor $d(\lambda)$ from the spectral theorem, which causes the presence of the factor of $d(\lambda)$ in the eigenvalue bounds on Euclidean space in the long range case. In contrast,  we have uniform Sobolev inequalities on $\mathbb{H}^n$ without the factor $d(\lambda)$ for $2(n + 1)/(n + 3) < p < 2$ on $\mathbb{H}^n$. Consequently, we have the better eigenvalue bounds on $\mathbb{H}^n$ in the long range case.

In summary, the following diagrams elucidate the relations between spectral measure estimates, Stein-Tomas estimates, Sobolev inequalities, and eigenvalue bounds.

$$\left.\begin{array}{ccc}\mbox{Stein-Tomas for $\frac{2(n + 1)}{n + 3} \leq p < 2$}&&\mbox{Stein-Tomas for $\frac{2n}{n + 2} \leq p \leq \frac{2(n + 1)}{n + 3}$}\\ \Uparrow&&\Uparrow\\
\mbox{Spectral measure estimates}&&\mbox{Spectral measure estimates}\\ \Downarrow&&\Downarrow\\ \mbox{Sobolev for $\frac{2(n + 1)}{n + 3} \leq p < 2$} &&  \mbox{Sobolev for $\frac{2n}{n + 2} \leq p \leq \frac{2(n + 1)}{n + 3}$}\\ \Downarrow&&\Downarrow\\ \mbox{Eigenvalue bounds for $\gamma \geq 1/2$} && \mbox{Eigenvalue bounds for $0 < \gamma < 1/2$}\end{array}\right.$$

The paper is organized as follows. We shall review the microlocal descriptions of the resolvent and spectral measure on (asymptotically) hyperbolic spaces in Section 2. We break the main theorem into a few propositions and then prove them in Section 3-5.  It is followed by the applications on eigenvalue bounds for Schr\"odinger operators with complex potentials.

The author is supported by EPSRC grant EP/R001898/1 and NSFC grant 11701094 while this work was in progress. The author would like to thank Zihua Guo, Andrew Hassell, and Katya Krupchyk for helpful discussions and comments. I am indebted to anonymous referees for their  insightful comments.

\section{Resolvent and spectral measure on hyperbolic spaces}

Consider the real hyperbolic space $\mathbb{H}^{n}$ and its Laplacian $\Delta$. There are several ways to view $\mathbb{H}^{n}$. We choose the Poincar\'e upper half plane model. Specifically,  $\mathbb{H}^{n}$ is the following Riemannian manifold, $$\bigg(\{(x, y) | x > 0, y \in \mathbb{R}^{n-1}\}, \frac{dx^2 + dy^2}{x^2} \bigg).$$ Then the Laplacian $\Delta$ reads $$ - (x\partial_x)^2 + (n-1) x \partial_x - \sum_{i = 1}^{n-1}(x\partial_{y_i})^2.$$ The spectrum of $\Delta$ consists only of the half line $[(n-1)^2/4, +\infty)$. So we denote $L = \Delta - (n-1)^2/4$ through out this paper for simplicity.

We are concerned about the resolvent. Specifically, the resolvent kernel of $(L - \lambda^2)^{-1}$, if $\lambda^2 \notin [0, \infty)$, reads \begin{equation}  \left\{\begin{array}{ll} \displaystyle \frac{C}{ \lambda} \bigg( \frac{1}{\sinh(\rho)} \frac{\partial}{\partial \rho}\bigg)^{(n-1)/2} e^{\pm \imath \lambda \rho}& \mbox{when $n$ is odd;}\\ \displaystyle C \int_\rho^\infty  e^{\pm \imath \lambda s} (\cosh(s) - \cosh(\rho))_+^{-(n-1)/2} \sinh(s) \,ds& \mbox{when $n$ is even,} \end{array}\right.\end{equation} where $\rho$ is the distance function. The resolvent $(L - \lambda^2)^{-1}$, as a map between some weighted $L^2$ spaces, extends to be meromorphic in the whole complex plane, with finite number of poles when $n$ is even. Therefore, its kernel is a  distribution  meromorphically extendible  in $\mathbb{C}$. In particular, the bottom of the spectrum, $\alpha = 0$, is neither an eigenvalue nor a resonance. Namely, it is analytic at this point.\footnote{This is also true on a broad class of asymptotically hyperbolic manifolds.}

Alternatively, Mazzeo-Melrose \cite{Mazzeo-Melrose} microlocally constructed the resolvent kernel on asymptotically hyperbolic spaces and described it as $$\text{Ker}(\Delta - \zeta(n - 1 - \zeta))^{-1} = R_{nd} + e^{\pm \imath \lambda \rho}R_{od}.$$ Here $R_{nd}$ is the kernel of a pseudodifferential operator of degree $-2$ supported near the diagonal, whilst $R_{od}$ is a smooth function supported away from the diagonal.

The analyticity at the bottom of the spectrum readily implies that (see \cite[Theorem 1.3]{Chen-Hassell2}) for sufficiently small $|\lambda|$,  \begin{equation}\label{eqn : low energy spectral measure}dE_{\sqrt{L}}(\lambda)(x, y, x', y') = (xx')^{(n-1)/2} \lambda ({(xx')}^{\imath \lambda } a(\lambda) - {(xx')}^{- \imath \lambda} a(-\lambda)),\end{equation} where $a \in C^\infty([-1, 1] \times \overline{\mathbb{H}^{n}} \times \overline{\mathbb{H}^{n}})$, $\overline{\mathbb{H}^n}$ is the compactification of $\mathbb{H}^n$,\footnote{From the viewpoint of asymptotically hyperbolic manifolds, it is more convenient to understand the compactification in the Poincar\'e disc model. We can map the upper half plane $\mathbb{R}^n_+$ to the unit ball $\mathbb{B}^n$ via the Cayley transform. The compactification $\overline{\mathbb{H}^{n}}$ we mean is just the closure $\overline{\mathbb{B}^{n}}$.} and $x$ is the boundary defining function of $\overline{\mathbb{H}^n}$.

To understand the uniform Sobolev inequalities, it is also important to understand the asymptotic behaviour of the resolvent for large spectral parameters. Following \cite{Mazzeo-Melrose},  Melrose-S\'{a} Barreto-Vasy \cite{Melrose-Sa Barreto-Vasy} constructed the following high energy resolvent on Cartan-Hadamard asymptotically hyperbolic manifolds.
\begin{theorem}\label{thm:semiclassical resolvent}The resolvent kernel $(L - \lambda^2)^{-1}$ is analytic in a neighbourhood of lower half plane $\{\Im \lambda \leq 0\}$ and takes the form $$\text{Ker}(L - \lambda^2)^{-1} = R_0 + e^{-\imath \lambda \rho}R_\infty, \quad \mbox{for $|\lambda| > 1$}$$ where $R_0$ is the kernel of a pseudodifferential operator of degree $-2$ supported near the diagonal, and $R_\infty$ is a smooth function pointwisely bounded by a multiple of $$\left\{ \begin{array}{cl}  \rho^{-(n-1)/2}(1 + |\lambda|)^{n/2-3/2}&\mbox{for $ \rho \leq C$}\\e^{-(n-1) \rho/2}(1 + |\lambda|)^{n/2-3/2}&\mbox{for $ \rho \geq C$.}\end{array}\right.$$\end{theorem}

As is well-known in spectral theory, one can obtain the spectral measure of the operator $\sqrt{L}$ through the resolvent near the spectrum. Using Melrose-S\'{a} Barreto-Vasy's resolvent construction, the  author and Hassell gave the microlocal description of the spectral measure at high energies and proved the bounds for the spectral measure in Theorem \ref{thm:kernelbounds}.

It is also useful to have the heat kernel. Davies-Mandouvalous \cite{Davies-Mandouvalos} proved that the heat kernel $e^{-t L}$ is equivalent to \begin{equation}\label{eqn : heat kernel}t^{- n/2} e^{-(n-1)\rho/2 - \rho^2/(4t)} (1 + \rho + t)^{n/2 - 3/2} (1 + \rho),\end{equation} where $\rho$ is the distance function. This has been further generalized to asymptotically hyperbolic spaces by Hassell and the author \cite{Chen-Hassell-heat kernel}.

%\section{The Kunze-Stein phenomenon}

%Let $\mathbb{G}$ be a connected semisimple Lie group of real rank $1$ with finite center, $\mathbb{K}$ a maximal compact subgroup.
%Cowling-Meda-Setti proved a sharp Kunze-Stein phenomenon in terms of Lorentz spaces. A special case for $L^p$ and $L^{p, \infty}$ reads that for $1 < p < 2$
%$$L^{p}(\mathbb{K} \backslash \mathbb{G}) \ast L^{p}(\mathbb{G}/\mathbb{K}) \subset L^{p, \infty}(\mathbb{K} \backslash \mathbb{G}/\mathbb{K}).$$ It, by duality, implies \begin{equation}L^p(\mathbb{G}/\mathbb{K}) \ast L^{p', 1}(\mathbb{K} \backslash \mathbb{G}/\mathbb{K}) \subset L^{p'}(\mathbb{G}/\mathbb{K}).\label{eqn : Kunze-Stein}\end{equation}

%In particular, if $\mathbb{G} = SO(n + 1, 1)$ and $\mathbb{K} = SO(n + 1)$, then $\mathbb{H}^{n + 1} = SO(n + 1, 1)/SO(n + 1)$. A direct consequence of \eqref{eqn : Kunze-Stein} on $\mathbb{H}^{n + 1}$ is that for a radial function $K(r)$ and $1 < p < 2$, $$\|f\ast K\|_{L^{p'}(\mathbb{H}^{n + 1})} \leq C \|f\|_{L^p(\mathbb{H}^{n + 1})} \|K\|_{L^{p', 1}(\mathbb{H}^{n + 1})}.$$

%When $K(r)$ is positive radial decreasing, we have $$\|K\|_{L^{p', 1}} \approx \int_0^1 K(r) r^{(n + 1)/p' - 1} dr + \int_1^\infty K(r) e^{nr/p'} dr.$$ Therefore, it yields the following convolution inequality\begin{equation}\label{eqn : Kunze-Stein radial decreasing}\|f\ast K\|_{L^{p'}(\mathbb{H}^{n + 1})} \leq C \|f\|_{L^p(\mathbb{H}^{n + 1})}\bigg(\int_0^1 K(r) r^{(n + 1)/p' - 1} dr + \int_1^\infty K(r) e^{nr/p'} dr\bigg)\end{equation}

\section{Sobolev estimates away from the spectrum }

 We proceed from the simplest case.
\begin{lemma}
For $\beta \leq 0$, \begin{equation}\label{eqn : sobolev for negative}\|(L - \beta)^{-1}\|_{L^q(\mathbb{H}^{n}) \rightarrow L^{q'}(\mathbb{H}^{n})} < C, \quad\quad\mbox{with $2n/(n + 2) \leq q < 2$}.\end{equation}
\end{lemma}

\begin{proof} We firstly assume $\beta < 0$ and invoke the heat kernel on $\mathbb{H}^{n}$ via the Laplace transform
$$(L - \beta)^{-1} = \mathcal{L}(e^{- \cdot L})(- \beta) = \int_0^\infty e^{- tL} e^{\beta t}\,dt.$$

In order to simplify the heat kernel \eqref{eqn : heat kernel} a bit, we break into the following three cases:
\begin{equation}\label{eqn : heat kernel bounds} \text{Ker} \,e^{-tL} \approx \left\{ \begin{array}{ll} t^{- n/2} e^{-(n-1)\rho/2 - \rho^2/(4t)} (1 + \rho )^{(n-1)/2}  & t < C \\t^{- n/2} e^{-(n-1)\rho/2 - \rho^2/(4t)} (1 + \rho )^{(n-1)/2}  &  \rho > t > C\\
t^{- 3/2} e^{-(n-1)\rho/2 - \rho^2/(4t)} (1 + \rho) & t > \rho, t > C.\end{array} \right.\end{equation} Using this, together with the integral formula $$\int_0^\infty t^{\nu - 1} e^{- \xi/t - \zeta t} \,dt = 2\bigg(\frac{\xi}{\zeta}\bigg)^{\nu/2} K_\nu(2\sqrt{\xi \zeta}), \quad \Re \xi >0,  \Re \zeta > 0$$ where $K_\nu$ is the modified Bessel functions,  we obtain that
$|\text{Ker}\,(L - \beta)^{-1}|$ with $\beta < 0$ is bounded from above by a multiple of
\begin{multline*} e^{-(n-1)\rho/2}\Big( (1 + \rho )^{(n-1)/2} \rho^{- (n - 2)/2}(-\beta)^{(n-2)/4} K_{(n - 2)/2}(\rho\sqrt{-\beta}) \\+ (1 + \rho)^{1/2}(-\beta)^{1/4} K_{1/2}(\rho\sqrt{-\beta})\Big).\end{multline*}
As is well-known, the modified Bessel function obeys the following asymptotic behaviours: $$K_\nu(h) \approx \left\{ \begin{array}{ll}C_\nu h^{-\nu}, & 0 < h < 1;  \\ C_\nu h^{-1/2} e^{-h}, & h > 1, \end{array}\right.$$ provided $\nu > 0$. Therefore, this yields that $$ |\text{Ker}\,(L -\beta)^{-1}| \leq \left\{\begin{array}{ll}C e^{-(n-1)\rho/2-\rho\sqrt{-\beta}}\Big(  (-\beta)^{(n-3)/4}  + 1\Big), &\rho\sqrt{-\beta} > 1,\\
C e^{-(n-1)\rho/2}(1 + \rho)^{(n-1)/2}\rho^{- (n - 2)}, & \rho\sqrt{-\beta} < 1.\end{array}\right.$$
 We further break up these bounds into six cases and consequently have the following upper bounds for the resolvent kernel $$ |\text{Ker}\,(L - \beta)^{-1}| \leq \left\{\begin{array}{ll}
C e^{-(n-1)\rho/2}   , &\mbox{Case (i)}: \rho > 1, \sqrt{-\beta} > 1,\\
C e^{-(n-1)\rho/2
}   , &\mbox{Case (ii)}: \rho > 1, \sqrt{-\beta} < 1, \rho \sqrt{-\beta} > 1,\\
C e^{-\rho\sqrt{-\beta}}  \langle -\beta\rangle^{(n-3)/4}  , &\mbox{Case (iii)}: \rho < 1, \sqrt{-\beta} > 1, \rho \sqrt{-\beta} > 1,\\
C \rho^{2-n}, &\mbox{Case (iv)}:  \rho < 1, \sqrt{-\beta} > 1, \rho \sqrt{-\beta} < 1,\\
C e^{-(n-1)\rho/2} \rho^{(3-n)/2}, &\mbox{Case (v)}:  \rho > 1, \sqrt{-\beta} < 1, \rho \sqrt{-\beta} < 1,\\
C \rho^{2-n}, &\mbox{Case (vi)}:  \rho < 1, \sqrt{-\beta} < 1.\end{array}\right.$$

Applying the Kunze-Stein phenomenon \eqref{KS} gives for all $\beta < 0$, \begin{equation}\label{eqn : constant resolvent bound away from spectrum}\|(L - \beta)^{-1}\|_{L^q(\mathbb{H}^{n}) \rightarrow L^{q'}(\mathbb{H}^{n})} < C, \quad\quad\mbox{with $2n/(n + 2) < q < 2$}.\end{equation} In fact, it suffices to check if the integral \begin{equation}\label{eqn : integral in KS}\int_0^\infty (\sinh \rho)^{n - 1} (1 + \rho) e^{- (n - 1)\rho/2} |\text{Ker} (L - \beta)^{-1}|^{1 + \varepsilon} \,d\rho\end{equation} is convergent uniformly in $\beta < 0$, where $q' = 2 + 2\varepsilon$ and $0 < \varepsilon < 2/(n - 2)$.

For Case (i) and Case (ii), the integral \eqref{eqn : integral in KS} reduces to $$\int_1^\infty  e^{(n - 1)\rho}(1 + \rho) e^{-(n - 1)\rho/2} \Big|e^{- (n - 1)\rho/2}\Big|^{1 + \varepsilon} \,d\rho,$$ which is obviously convergent.

For Case (iv) and Case (vi), we have to estimate the fractional integral $$\int_0^1  \rho^{n - 1}\rho^{(2 - n) + \varepsilon(2 - n)} \,d\rho,$$ which is convergent as long as $\varepsilon < 2/(n - 2)$.

For Case (iii), the integral is bounded above by $$\int_{(-\beta)^{-1/2}}^1 \rho^{n - 1}   \Big|e^{-\rho\sqrt{-\beta}} (\sqrt{-\beta})^{(n-3)/2}\Big|^{1 + \varepsilon} \,d\rho.$$ The range of $\varepsilon$ guarantees the power of $\rho$ is bigger than that of $\sqrt{-\beta}$. Since $\rho < 1$ and $$e^{-\rho\sqrt{-\beta}} (\rho\sqrt{-\beta})^{(n-3)(1 + \varepsilon)/2} < C,$$ we have the uniform convergence in $\beta$.

For Case (v), we end up with $$\int_1^\infty  e^{(n - 1)\rho}(1 + \rho) e^{-(n - 1)\rho/2} \Big|e^{- (n - 1)\rho/2}\rho^{(3-n)/2}\Big|^{1 + \varepsilon} \,d\rho.$$ We will have $e^{-(n-1)\varepsilon\rho/2}$ left in the integrand after trivial cancellations. This factor kills any polynomial growth in $\rho$ and makes the integral uniformly convergent.

This inequality also holds at the endpoint $q = 2n/(n + 2)$, which corresponds to $\varepsilon = 2/(n - 2)$. In cases (i) (ii) (iii) (v), the proof above works verbatim. But it fails in cases (iv) (vi) when $\rho$ is small. To remedy this, one can apply the Hardy-Littlewood-Sobolev inequality to the diagonal part of the resolvent kernel, since the diagonal part is not affected by the exponential volume growth on hyperbolic spaces.

Noting the operator norm bound is independent of $\beta$, we have the inequality \eqref{eqn : sobolev for negative} in addition at $\beta = 0$. The proof is now complete.\end{proof}

This is further generalized to spectral parameters away from the positive real axis.

\begin{lemma}

For any $\alpha \in \mathbb{C}$ with $|\arg(\alpha)| \geq \theta > 0$, we have \begin{equation}\label{eqn : sobolev away from spectrum with constant bound}\|(L - \alpha)^{-1}\|_{L^q(\mathbb{H}^{n}) \rightarrow L^{q'}(\mathbb{H}^{n})} < C, \quad\quad\mbox{with $2n/(n + 2) \leq q < 2$}.\end{equation}

 \end{lemma}

\begin{proof}

 By duality, \eqref{eqn : sobolev for negative} implies that for $\beta < 0$\begin{eqnarray*}\|(L - \beta)^{-1/2}\|_{L^{q}(\mathbb{H}^{n}) \rightarrow L^{2}(\mathbb{H}^{n})} &<& C,\\ \|(L - \beta)^{-1/2}\|_{L^{2}(\mathbb{H}^{n}) \rightarrow L^{q'}(\mathbb{H}^{n})} &<& C.\end{eqnarray*}
  For any complex spectral parameter  $\alpha = \beta + \imath \gamma$ with  $\gamma \neq 0$,  it follows that $$\|(L - \alpha)^{-1}\|_{L^q(\mathbb{H}^{n}) \rightarrow L^{q'}(\mathbb{H}^{n})}  < C \|(L - \beta)(L - \alpha)^{-1}\|_{L^2(\mathbb{H}^{n}) \rightarrow L^{2}(\mathbb{H}^{n})}.$$ The right hand side is indeed bounded from above $$\|(L - \beta)(L - \alpha)^{-1}\|_{L^2(\mathbb{H}^{n}) \rightarrow L^{2}(\mathbb{H}^{n})} = \sup_{\lambda^2 > 0} \frac{|\lambda^2 - \beta|}{|\lambda^2 - \beta - \imath \gamma|} \leq 1.$$ Similarly, for $\alpha = - \beta + \imath \gamma$ with $|\arg \alpha| \geq \theta > 0$, we have $$\|(L - \beta)(L - \alpha)^{-1}\|_{L^2(\mathbb{H}^{n}) \rightarrow L^{2}(\mathbb{H}^{n})} = \sup_{\lambda^2 > 0} \frac{|\lambda^2 - \beta|}{|\lambda^2 + \beta - \imath \gamma|} \leq C (1 + \bigg|\frac{\beta}{\gamma}\bigg| ) \leq C.$$
\end{proof}

Combining this with $L^2$ bound of the resolvent,  we can prove \begin{proposition}\label{prop : away spectrum}If $\alpha \in \mathbb{C}$ is  away from an open sector $\{\alpha : |\arg(\alpha)| < \theta\}$ containing the positive real axis, we have \begin{equation}\label{eqn : sobolev away from spectrum}\|(L - \alpha)^{-1}\|_{L^{q}(\mathbb{H}^{n}) \rightarrow L^{q'}(\mathbb{H}^{n}) } \leq C |\alpha|^{n(1/q - 1/2) - 1}, \end{equation}for $2n/(n + 2) \leq q < 2$.
	
	 Alternatively, for any $\alpha \in \mathbb{C}$, we have \begin{equation}\label{eqn : sobolev away from spectrum with cutoff}\|(1 - \phi(L/|\alpha|))(L - \alpha)^{-1}\|_{L^{q}(\mathbb{H}^{n}) \rightarrow L^{q'}(\mathbb{H}^{n}) } \leq C |\alpha|^{n(1/q - 1/2) - 1},\end{equation} provided that $\phi \in C_c^\infty(\mathbb{R})$ is a cut-off function supported around $1$ and equal to $1$ in a small compact neighbourhood of $1$.\end{proposition}

We remark that the insertion of the cut-off function $(1 - \phi(\cdot/|\alpha|))$ eliminates the restriction of the spectral parameter $\alpha$ but  does not change the off-spectrum nature. The constraint of $\alpha$ being away from the spectrum in \eqref{eqn : sobolev away from spectrum} is moved to the cut-off function in \eqref{eqn : sobolev away from spectrum with cutoff}. The support of the cut-off function is determined by the sector $\{\alpha : |\arg(\alpha)| < \theta\}$.

\begin{proof}
 First of all, spectral theorem yields that for any $\alpha \in \mathbb{C}$ with  $|\arg(\alpha)| \geq \theta > 0$
$$\|(L - \alpha)^{-1}\|_{L^{2}(\mathbb{H}^{n}) \rightarrow L^{2}(\mathbb{H}^{n})} \leq  \sup_{\lambda^2 > 0} C / |\lambda^2 - \alpha| \leq C/|\alpha|.$$

Then \eqref{eqn : sobolev away from spectrum} is obtained by applying Riesz-Thorin interpolation to this $L^2$ bound and \eqref{eqn : sobolev away from spectrum with constant bound}.

Noting that the operator $1 - \phi(L/|\alpha|)$ is commutative with the function of $L$, we can use the same argument for \eqref{eqn : sobolev away from spectrum with constant bound} to prove, for any $\alpha \in \mathbb{C}$, \begin{eqnarray*}\lefteqn{\|(1 - \phi(L/|\alpha|))(L - \alpha)^{-1}\|_{L^q(\mathbb{H}^{n}) \rightarrow L^{q'}(\mathbb{H}^{n})}}\\
&<& C\|(1 - \phi(L/|\alpha|))(L + |\alpha|)(L - \alpha)^{-1}\|_{L^2(\mathbb{H}^{n}) \rightarrow L^{2}(\mathbb{H}^{n})}\\
&<& C \sup_{\lambda^2 > 0}   \frac{|1 - \phi(\lambda^2/|\alpha|)|  (\lambda^2 + |\alpha|)}{|\lambda^2 - \alpha|}\\
&<& C \sup_{\mu \geq 0}    \frac{|1 - \phi(\mu)| (\mu + 1)}{|\mu - e^{\imath \arg \alpha}|}\\
&<& C. \end{eqnarray*} The last inequality used the fact that $|\mu - e^{\imath \arg \alpha}|$ is bounded uniformly in $\alpha$ from below on the support of $1 - \phi(\mu)$.

Finally, the interpolation with $$\|(1 - \phi(L/|\alpha|))(L - \alpha)^{-1}\|_{L^2(\mathbb{H}^{n}) \rightarrow L^2(\mathbb{H}^{n})} \leq C/|\alpha|$$ yields \eqref{eqn : sobolev away from spectrum with cutoff}.

\end{proof}

\section{Sobolev estimates near the spectrum}

It remains to consider the case when $\alpha$ is near the spectrum, say $ |\arg(\alpha)| < \theta < \pi/4$.  By choosing a cut-off function $\phi$ supported around $1$ as above, we consider $\phi(L/|\alpha|)(L - \alpha)^{-1}$ instead. This spectral cut-off does not change the result but brings us some convenience to use the spectral measure estimates.

On the one hand,  one can prove the following low energy results \begin{proposition}\label{prop : low near spectrum} Suppose $\alpha \in \mathbb{C} \setminus \mathbb{R}$ with $|\alpha| < 1$ and $|\arg(\alpha)| < \theta < \pi/4$. For $2n/(n + 2) \leq q < 2$, $$ \|\phi(L/|\alpha|)(L - \alpha)^{-1}\|_{L^{q}(\mathbb{H}^{n}) \rightarrow L^{q'}(\mathbb{H}^{n})} \leq C.$$\end{proposition}

\begin{proof}

As is indicated in the description of Mazzeo-Melrose, the resolvent kernel near the diagonal is a pseudodifferential operator of degree $-2$. Then we have a kernel bound $$\text{Ker}\Big(\phi(L/|\alpha|)(L - \alpha)^{-1}\Big)  = O(\rho^{- n + 2}),\quad \mbox{for small $\rho$}.$$

If we live away from the diagonal, we rewrite this kernel in terms of the spectral theorem, $$\text{Ker}\Big(\phi(L/|\alpha|)(L - \alpha)^{-1}\Big) = \int \phi(\lambda^2/|\alpha|) (\lambda^2 - \alpha)^{-1} dE_{\sqrt{L}}(\lambda) d\lambda.$$  Then \eqref{eqn : low energy spectral measure} yields that
\begin{multline*}\lefteqn{\text{Ker}\Big(\phi(L/|\alpha|)(L - \alpha)^{-1}\Big)}\\ =  (xx')^{-(n - 1)/ 2} \int \phi(\lambda^2/|\alpha|)\frac{\lambda}{\lambda^2 - \alpha} \big(e^{\imath \lambda \log(xx')} a(\lambda) - e^{- \imath \lambda \log(xx')} a(-\lambda)\big) d\lambda\\ =  \frac{(xx')^{-(n - 1)/ 2}}{2} \int \phi(\lambda^2/|\alpha|)\bigg(\frac{1}{\lambda - \sqrt{\alpha}} + \frac{1}{\lambda + \sqrt{\alpha}}\bigg)  \sum_\pm \pm  e^{\pm\imath \lambda \log(xx')} a(\pm\lambda) d\lambda.\end{multline*}

For $\alpha \in \mathbb{C} \setminus \mathbb{R}$ with $|\arg(\alpha)| < \theta$, we write $\sqrt{\alpha} = \beta + \imath \epsilon$. Shifting $\lambda$ by $\pm \beta$ gives \begin{multline*}\lefteqn{\text{Ker}\Big(\phi(L/|\alpha|)(L - \alpha)^{-1}\Big)} \\= 
	\frac{(xx')^{-(n - 1)/ 2}}{2} e^{ \imath \log(xx') \beta}  \int \frac{\phi((\lambda+\beta)^2/|\alpha|)}{\lambda - \imath \epsilon}    e^{\imath \lambda \log(xx')} a(\lambda + \beta) d\lambda\\ 
	-	\frac{(xx')^{-(n - 1)/ 2}}{2} e^{-\imath \log(xx') \beta}  \int \frac{\phi((\lambda+\beta)^2/|\alpha|)}{\lambda - \imath \epsilon}  e^{-\imath \lambda \log(xx')} a(-\lambda - \beta) d\lambda\\ 
	+	\frac{(xx')^{-(n - 1)/ 2}}{2} e^{- \imath \log(xx') \beta}  \int \frac{\phi((\lambda-\beta)^2/|\alpha|)}{\lambda + \imath \epsilon}    e^{\imath \lambda \log(xx')} a(\lambda - \beta) d\lambda\\ 
	-	\frac{(xx')^{-(n - 1)/ 2}}{2} e^{\imath \log(xx') \beta}  \int \frac{\phi((\lambda-\beta)^2/|\alpha|)}{\lambda + \imath \epsilon}  e^{-\imath \lambda \log(xx')} a(-\lambda + \beta) d\lambda.\end{multline*} 
Now we  view the products $\phi((\lambda \pm \beta)^2/|\alpha|) a(\pm \lambda \pm \beta)$ as compactly supported smooth functions in $\lambda$. Then their Fourier transforms are Schwartz. The factors $1/(\lambda \pm \imath \epsilon)$ can be Fourier transformed by $$\mathcal{F}( e^{- c z} H(z) )(\zeta) = \frac{1}{c + \imath \zeta}, $$ where $c > 0$ and $H(\cdot)$ is the Heaviside unit step function. Hence above integrals are all convolutions of some Schwartz function and $e^{- |\epsilon| \cdot} H(\cdot)$. Invoking  $\rho \sim -\log(xx')$ if $\rho$ is large, we conclude that  $$\bigg|\text{Ker}\Big(\phi(L/|\alpha|)(L - \alpha)^{-1}\Big)\bigg| \leq C 
e^{-(n - 1)\rho/ 2}.$$ 

%Since $|\arg(\alpha)| < \theta < \pi/4$, it follows that $$|\sin(\arg(\sqrt{\alpha}))| = \frac{|\epsilon|}{|\sqrt{\alpha}|} \leq \sin(\theta/2).$$ By $|\alpha| < 1$, we see the kernel bound $$\bigg|\text{Ker}\Big(\phi(L/|\alpha|)(L - \alpha)^{-1}\Big)\bigg| \leq C  e^{-(n - 1)\rho/ 2}  e^{\sin(\theta/2)\rho}.$$ 

Finally, we complete the proof by applying \eqref{KS} to the kernel bound of $\phi(L/|\alpha|)(L - \alpha)^{-1}$. \end{proof}%However, the term $e^{\sin (\theta/2) \rho}$ grows exponentially at infinity and cancels out some of the decay from $e^{-(n - 1) \rho/2}$. But for any $q < 2$, we can choose $\theta$ such that $\sin(\theta/2) < (n - 1)(1/2 - 1/q')$ and then apply \eqref{KS} to the kernel bound of $\phi(L/|\alpha|)(L - \alpha)^{-1}$, which proves the proposition.

On the other hand,  we prove the following high energy results
\begin{proposition} For $2(n + 1)/(n + 3) \leq q < 2$ and $|\alpha|>1$ with $|\arg(\alpha)| < \theta$, \begin{equation}\label{eqn : q to q'}\|\phi(L/|\alpha|)(L - \alpha)^{-1}\|_{L^{q}(\mathbb{H}^{n}) \rightarrow L^{q'}(\mathbb{H}^{n })} \leq C |\alpha|^{1/2 - 1/q}.\end{equation} For $2n/(n + 2) \leq q \leq 2(n + 1)/(n + 3)$ and $|\alpha|>1$ with $|\arg(\alpha)| < \theta$,\begin{equation}\label{eqn : q-q' (normal range)}\|\phi(L/|\alpha|)(L - \alpha)^{-1}\|_{L^q(\mathbb{H}^{n})\rightarrow L^{q'}(\mathbb{H}^{n})} \leq C |\alpha|^{n(1/q - 1/2) - 1}.\end{equation}\end{proposition}

\begin{proof} The estimates \eqref{eqn : q-q' (normal range)} is the same with the Euclidean case in \cite{Guillarmou-Hassell}. We can use the proof in there verbatim to show \eqref{eqn : q-q' (normal range)}.
 
The strategy to prove \eqref{eqn : q to q'} is to use the Stein's complex interpolation \cite{interpolation} for the analytic family of operators $H_{s, \alpha}(\sqrt{L/|\alpha|})$, where $$H_{s, \alpha}(x)  = e^{s^2} |\alpha|^s \phi(x^2) (1 - x^2 + \imath 0)^s,$$ provided $\phi$ is a cut-off function supported around $1$. 

In particular, it suffices to prove $(L^q, L^{q'})$ estimates, with some $2(n + 1)/(n + 3) \leq q < 2$, for $$H_{-1, \alpha}(\sqrt{L/|\alpha|}) = e \text{Ker} \phi(L/|\alpha|) (\alpha - L + \imath 0)^{-1},$$ for $\alpha > 1$.
Here we are only concerned about the estimates at the spectrum as the off-spectrum part is less singular.

By complex interpolation, it suffices to establish $(L^2, L^2)$ estimates for $H_{\imath t, \alpha}(\sqrt{L/|\alpha|})$ and $(L^1, L^\infty)$ estimates for
$H_{- j - 1 + \imath t, \alpha}(\sqrt{L/|\alpha|})$ with any integer $j > (n-1)/2$. As the former is trivial, it suffices to prove an upper bound of the kernel of $H_{- j - 1 + \imath t, \alpha}(\sqrt{L/|\alpha|})$.

We firstly rewrite the kernel of $H_{- j - 1 + \imath t, \alpha}(\sqrt{L/|\alpha|})$, in terms of the spectral measure, as $$e^{(-j -1+\imath t)^2}|\alpha|^{-j+\imath t - 1} \int_0^\infty \phi\Big(\frac{\lambda^2}{|\alpha|}\Big) \bigg(1 - \frac{\lambda^2}{|\alpha|} + \imath 0\bigg)^{- j -1 + \imath t} dE_{\sqrt{L}}(\lambda) d\lambda.$$
To make use of the spectral measure upper bound \eqref{eqn : spectral measure upper bound}, we change coordinates and further rewrite it as
\begin{eqnarray*}\lefteqn{H_{- j - 1 + \imath t, \alpha}(\sqrt{L/|\alpha|})}\\ &=&
e^{(-j -1+\imath t)^2} |\alpha|^{-j+\imath t - 1/2} \int_0^\infty \phi(\lambda^2) (1 - \lambda^2  + \imath 0)^{- j -1 + \imath t} dE_{\sqrt{L}}(\sqrt{\alpha}\lambda) d\lambda
\\ &=&
e^{(-j -1+\imath t)^2} |\alpha|^{-j+\imath t - 1/2} \int_0^\infty \phi(\lambda) (1 - \lambda + \imath 0)^{- j -1 + \imath t} dE_{\sqrt{L}}(\sqrt{\alpha}\sqrt{\lambda}) \frac{d\lambda}{2\sqrt{\lambda}}
\\&=& \frac{e^{(-j-1+\imath t)^2} |\alpha|^{-j-1/2+\imath t}}{\prod_{k = 0}^{j-2}(-j - 1 + k + \imath t)} \int_0^\infty \phi(\lambda) \frac{d^{j - 1}}{d\lambda^{j - 1}}\Big((1 - \lambda + \imath 0)^{- j -1 + \imath t}\Big) dE_{\sqrt{L}}(\sqrt{\alpha}\sqrt{\lambda}) \frac{d\lambda}{2\sqrt{\lambda}}.\end{eqnarray*}
We now make integration by parts and estimate the kernel as follows.
\begin{eqnarray*}\lefteqn{(-1)^{j-1} H_{- j - 1 + \imath t, \alpha}(\sqrt{L/|\alpha|})}
\\ &=& \frac{e^{(-j-1+\imath t)^2} |\alpha|^{-j-1/2+\imath t}}{\prod_{k = 0}^{j-2}(-j - 1 + k + \imath t)} \int_0^\infty  (1 - \lambda  + \imath 0)^{- 2 + \imath t}\frac{d^{j - 1}}{d\lambda^{j - 1}}\bigg(\frac{\phi(\lambda)}{2\sqrt{\lambda}} dE_{\sqrt{L}}(\sqrt{\alpha}\sqrt{\lambda}) \bigg) d\lambda.\end{eqnarray*}

Now we use a family of distributions $\{\chi_+^a = x_+^a/\Gamma(a + 1)\}$ defined on all $a \in \mathbb{C}$, where $\Gamma$ is the gamma function and $$x_+^a = \left\{ \begin{array}{ll}x^a & \mbox{if $x \geq 0$}\\0 & \mbox{if $x < 0$} \end{array} \right..$$ It follows that $\chi_+^0(x) = H(x)$ and $\chi_+^{-k} = \delta_0^{(k - 1)}$, where $H$ is the Heaviside function and $\delta_0^{(k - 1)}$ is the $k-1$-derivative of the delta function at $0$. By \cite[p.608, (26)]{Guillarmou-Hassell},  the integral above obeys \begin{eqnarray*}
\lefteqn{\bigg|\int_0^\infty  (1 - \lambda  + \imath 0)^{- 2 + \imath t}\frac{d^{j - 1}}{d\lambda^{j - 1}}\bigg(\frac{\phi(\lambda)}{2\sqrt{\lambda}} dE_{\sqrt{L}}(\sqrt{\alpha}\sqrt{\lambda}) \bigg) d\lambda\bigg|}
\\ &\leq& C (1 + |t|) e^{\pi t/ 2}
\sup_{\sigma}\bigg|\bigg(\chi_+^{-1}  \ast \frac{d^{j - 1}}{d\lambda^{j - 1}}\bigg(\frac{\phi(\lambda)}{2\sqrt{\lambda}} dE_{\sqrt{L}}(\sqrt{\alpha}\sqrt{\lambda}) \bigg)\bigg)(\sigma) \bigg|^{1/2}\\ &&\quad\quad\quad\quad\quad\quad\,
\sup_{\sigma}\bigg|  \bigg( \chi_+^{-3}  \ast \frac{d^{j - 1}}{d\lambda^{j - 1}}\bigg(\frac{\phi(\lambda)}{2\sqrt{\lambda}} dE_{\sqrt{L}}(\sqrt{\alpha}\sqrt{\lambda}) \bigg) \bigg) (\sigma)\bigg|^{1/2}
\\ &\leq& C (1 + |t|) e^{\pi t/ 2}
\sup_{\sigma}\bigg|\int_0^\infty \delta(\sigma - \lambda)   \frac{d^{j - 1}}{d\lambda^{j - 1}}\bigg(\frac{\phi(\lambda)}{2\sqrt{\lambda}} dE_{\sqrt{L}}(\sqrt{\alpha}\sqrt{\lambda}) \bigg) d\lambda\bigg|^{1/2}\\ &&\quad\quad\quad\quad\quad\quad\,
\sup_{\sigma}\bigg|\int_0^\infty  \delta^{(2)}(\sigma - \lambda)    \frac{d^{j - 1}}{d\lambda^{j - 1}}\bigg(\frac{\phi(\lambda)}{2\sqrt{\lambda}} dE_{\sqrt{L}}(\sqrt{\alpha}\sqrt{\lambda}) \bigg) d\lambda\bigg|^{1/2}
\\ &\leq& C (1 + |t|) e^{\pi t/ 2}
\sup_{\sigma}\bigg|  \frac{d^{j - 1}}{d\lambda^{j - 1}}\bigg(\frac{\phi(\lambda)}{2\sqrt{\lambda}} dE_{\sqrt{L}}(\sqrt{\alpha}\sqrt{\lambda}) \bigg)\bigg|_{\lambda = \sigma} \bigg|^{1/2}\\ &&\quad\quad\quad\quad\quad\quad\,
\sup_{\sigma}\bigg|  \frac{d^{j + 1}}{d\lambda^{j + 1}}\bigg(\frac{\phi(\lambda)}{2\sqrt{\lambda}} dE_{\sqrt{L}}(\sqrt{\alpha}\sqrt{\lambda}) \bigg)\bigg|_{\lambda = \sigma}  \bigg|^{1/2}.
\end{eqnarray*}

Recall that $\phi$ is supported around $1$. The spectral measure estimates \eqref{eqn : spectral measure upper bound} yields that
$$|H_{- j - 1 + \imath t, \alpha}(\sqrt{L/|\alpha|})| \leq C_{j, t} |\alpha|^{- j - 1/2 + (n-1)/4 + j/2} \leq C_{j, t} |\alpha|^{-1/2}.$$

Noting that $\theta = 1/(j + 1), q = 2(j + 1)/(j + 2)$ solves the elementary system of equations, $$\left\{ \begin{array}{l} 0(1 - \theta) + (j + 1) \theta = 1\\ (1-\theta)/2 + \theta = 1/q\end{array} \right.,$$ we obtain that for any $1/2 < \epsilon < 1$, $$ \|\phi(L/|\alpha|)(L - \alpha + \imath 0)^{-1}\|_{L^q(\mathbb{H}^{n}) \rightarrow L^{q'}(\mathbb{H}^{n})} \leq C |\alpha|^{1/2 - 1/q}.$$

\end{proof}

\section{Eigenvalue bounds for Schr\"odinger operators with complex potentials}

Inspired by the work of Frank-Simon \cite{Frank-Simon II, Frank-Simon III} , we apply the Sobolev inequalities on hyperbolic space to the study of eigenvalue bounds for a Schr\"odinger operator with complex potential. More precisely, for an $L^2$-eigenvalue $\lambda$ of a Schr\"odinger operator, $\Delta + V$, on $\mathbb{H}^{n + 1}$ with a complex potential $V$, we want to understand the upper bound of $|\lambda|$ in terms of the Lebesgue norm of $|V|$. In addition, we will prove that if the norm of the potential is sufficiently small the Schr\"odinger operator has no eigenvalues.

Let $\lambda \in \mathbb{C}$ be an eigenvalue and $\psi \in H^1(\mathbb{H}^{n})$ the corresponding eigenfunction of $L + V$, namely $$(L + V) \psi = \lambda \psi.$$

\subsection*{Short range $0 < \gamma \leq 1/2$}
\begin{proof}[Proof of Theorem \ref{thm : short range}]
Assume first that $\lambda \in \mathbb{C} \setminus [0, \infty)$. We write $$\gamma + \frac{n}{2} = \frac{p}{2-p}.$$ This implies $2n/(n + 2) < p \leq 2(n+1)/(n+3)$ and $2(n + 1)/(n - 1) \leq p' < 2n/(n - 2)$. Proposition \ref{prop : m-sectorial} shows that $\Delta + (V - n^2/4)$ is an m-sectorial operator with domain in $H^1(\mathbb{H})$. By Sobolev's embedding, we further have $\psi \in L^{2n/(n - 2)}(\mathbb{H})$, whence $\psi \in L^r(\mathbb{H})$ for $2 \leq r \leq 2n/(n - 2)$. Additionally, we have $$\psi = (L - \lambda)^{-1}(L - \lambda) \psi = - (L - \lambda)^{-1} (V\psi).$$

Using H\"older's inequality and uniform Sobolev inequalities  \eqref{eqn : q-q' sobolev all spectral parameters (normal range)}, we obtain that \begin{eqnarray*}
\|\psi\|_{L^{p'}(\mathbb{H}^n)} &\leq & \|(L - \lambda)^{-1}\|_{L^p(\mathbb{H}^n) \rightarrow  L^{p'}(\mathbb{H}^n)} \|V\psi\|_{L^p(\mathbb{H}^n)}\\
&\leq& C |\lambda|^{n(1/p - 1/p') - 1} \|V\|_{L^{\gamma + n/2}(\mathbb{H}^n)} \|\psi\|_{L^{p'}(\mathbb{H}^n)}.
\end{eqnarray*}

Noting that $$\frac{n}{2} \bigg(\frac{1}{p} - \frac{1}{p'}\bigg) - 1 = - \frac{\gamma}{\gamma + n/2},$$ we conclude that \begin{equation}
\label{eqn : eigenvalue bound 1} |\lambda|^\gamma \leq C \|V\|^{\gamma + n/2}_{L^{\gamma + n/2}(\mathbb{H}^n)}.
\end{equation}

If $\lambda \in (0, \infty)$, we instead consider $$\psi_\epsilon = (L - \lambda - \imath \epsilon)^{-1} (L - \lambda) \psi = f_\epsilon(L) \psi, \quad  \mbox{with $f_\epsilon(t) = (t - \lambda)/(t - \lambda - \imath \epsilon)$ for $t > 0$}.$$ Then the spectral theorem yields that $$\|\psi_\epsilon - \psi\|^2_{L^2(\mathbb{H}^n)} = \|f_\epsilon(L) \psi - \psi\|^2_{L^2(\mathbb{H}^n)} = \int |f_\epsilon(t) - 1|^2 d(E_L(t)\psi, \psi)_{L^2(\mathbb{H}^n)},$$
where $E_L(t)$ is the spectral projection of $L$. In view of the fact that $f_\epsilon \rightarrow 1$ as $\epsilon \rightarrow 0$, the dominated convergence theorem yields that $\psi_\epsilon \rightarrow \psi$ in $L^2(\mathbb{H}^n)$ as $\epsilon \rightarrow 0$.

For $\psi_\epsilon$, we similarly obtain that
$$\|\psi_\epsilon\|_{L^{p'}(\mathbb{H}^n)} \leq C|\lambda|^{n(1/p - 1/p')/2 - 1} \|V\|_{L^{\gamma + n/2}(\mathbb{H}^{n})} \|\psi\|_{L^{p'}(\mathbb{H}^n)}.$$ It follows that there exists $\tilde{\psi} \in L^{p'}(\mathbb{H}^n)$ such that $\psi_\epsilon \rightarrow \tilde{\psi}$ in the weak $\ast$ topology of $L^{p'}(\mathbb{H}^n)$, whence $\psi = \tilde{\psi} \in L^{p'}(\mathbb{H}^n)$. Consequently, we have $$\|\psi\|_{L^{p'}(\mathbb{H}^n)} \leq \liminf_{\epsilon \rightarrow 0} \|\psi_\epsilon\|_{L^{p'}(\mathbb{H}^n)} \leq |\lambda|^{n(1/p - 1/p') - 1} \|V\|_{L^{\gamma + n/2}(\mathbb{H}^n)} \|\psi\|_{L^{p'}(\mathbb{H}^n)},$$ which completes the proof for the short range case.
\end{proof}
\subsection*{Long range $\gamma > 1/2$}

\begin{proof}[Proof of Theorem \ref{thm : long range}]
We still use the proof for the short range case but replace \eqref{eqn : q-q' sobolev all spectral parameters (normal range)} by \eqref{eqn : q-q' sobolev all spectral parameters}. This leads to that  \begin{eqnarray*}\|\psi\|_{L^{p'}(\mathbb{H}^n)} &\leq& \|(L - \lambda)^{-1}\|_{L^p(\mathbb{H}^n) \rightarrow L^{p'}(\mathbb{H}^n)} \|V\psi\|_{L^p(\mathbb{H}^n)}\\ &\leq& C |\lambda|^{1/2 - 1/p} \|V\|_{L^{\gamma + n/2}(\mathbb{H}^n)} \|\psi\|_{L^{p'}(\mathbb{H}^n)}. \end{eqnarray*}

Combing this together with $$\frac{1}{p} = \frac{1 + \gamma + n/2}{2(\gamma + n/2)},$$ we conclude that, $$|\lambda|^{1/2} \leq C \|V\|_{L^{\gamma + n/2}}^{\gamma + n/2}.$$
\end{proof}
\subsection*{Sufficiently 'small' potential $V$}
\begin{proof}[Proof of Theorem \ref{thm : no eigenvalue}]
In the end, we prove, by a contradiction argument, that if the norm $\|V\|_{L^{\gamma + n/2}(\mathbb{H}^n)}$ is sufficiently small, the Schr\"odinger operator $\Delta + V$ has no eigenvalues.

Let $\lambda$ be an eigenvalue of $L + V$ with $\|V\|_{L^{\gamma + n/2}(\mathbb{H}^n)} < c < 1$. It follows, from the eigenvalue bounds we have proved, that  $|\lambda| < C c^{(\gamma + n/2)/\gamma} < C$ for $\gamma \geq 0$. When $\lambda \notin [0, \infty)$, we have that for any eigenfunction $\psi$ of $\lambda$ \begin{eqnarray*}\|\psi\|_{L^{p'}(\mathbb{H}^n)} &\leq& \|(L-\lambda)^{-1}\|_{L^p(\mathbb{H}^n) \rightarrow L^{p'}(\mathbb{H}^n)} \|V\psi\|_{L^p(\mathbb{H}^n)}\\
&\leq& C_1 \|V\|_{L^{\gamma + n/2}(\mathbb{H}^n)} \|\psi\|_{L^{p'}(\mathbb{H}^n)},\end{eqnarray*} where we invoked the low energy Sobolev inequalities \eqref{eqn : q-q' sobolev small spectral parameter}.  When $\lambda \in (0, \infty)$, it is similar to deduce that \begin{eqnarray*}\|\psi\|_{L^{p'}(\mathbb{H}^n)}&\leq&\liminf_{\epsilon \rightarrow 0}\|\psi_\epsilon\|_{L^{p'}(\mathbb{H}^n)}\\ &\leq& \|(L-\lambda - \imath\epsilon)^{-1}\|_{L^p(\mathbb{H}^n) \rightarrow L^{p'}(\mathbb{H}^n)} \|V\psi\|_{L^p(\mathbb{H}^n)}\\
&\leq& C_2 \|V\|_{L^{\gamma + n/2}(\mathbb{H}^n)} \|\psi\|_{L^{p'}(\mathbb{H}^n)},\end{eqnarray*} where $\psi_\epsilon = (L -\lambda -\imath\epsilon)^{-1}(L-\lambda)\psi$ as above.
If $c < \min\{1/C_1, 1/C_2\}$, these inequalities fail to hold. Therefore, there are no eigenvalues.
\end{proof}
%We prove\begin{theorem}\label{thm : eigenvalue bounds} If $E$ is an eigenvalue with $|E| > 1$ $$|E|^{1/2} \leq C\int_{\mathbb{H}^{n + 1}} |V|^{p/(2-p)} .$$\end{theorem}
%\begin{proof}

%This theorem generically reduces to Sobolev inequalities. In fact, Frank-Simon proved that \begin{proposition}Let $X$ be a separable complex Banach space of functions such that $L^2 \cap X$ is dense in $X$ and the duality pairing $X^\ast \times X \rightarrow \mathbb{C}$ extends the inner product in $L^2$. If Sobolev inequality, $$\|(\Delta - z)^{-1}\|_{X \rightarrow X^\ast} \leq N(z),$$ where $N(z)$ is finite for $z \in \mathbb{C} \setminus [0,\infty)$ and continuous up to $[0, \infty) \setminus I$ for some $I \subset [0, \infty)$. For any $L^2$-eigenvalue $E \in \mathbb{C} \setminus I$ of $\Delta + V$, we have \begin{equation}1 \leq N(E) \|V\|_{X^\ast \rightarrow X}.\label{eqn : frank-simon ineq}\end{equation} \end{proposition}

%Let $\psi$ be the eigenfunction of $E$. Since $\Delta + V$ is an $m$-sectorial operator with domain $H^1$, we have $\psi \in H^1$. (!!!! Check!!!!) Then the Sobolev embedding and interpolation yields $\psi \in L^{p'}$. Applying the high energy Sobolev inequalities and H\"older's inequalities, we obtain $$N(E)^{-1} \leq C \|V\|_{L^{p'}(\mathbb{H}^{n + 1})\rightarrow L^p(\mathbb{H}^{n + 1})}  \leq C \|V\|_{L^{p/(2-p)}},$$ which proves the theorem.\end{proof}

\begin{appendix}

\section{M-Sectorial operators}

The purpose of this section is to prove that
\begin{proposition}\label{prop : m-sectorial}Given a complex potential $V \in L^p(\mathbb{H}^n)$ with $n/2 \leq p < \infty$, the Schr\"odinger operator $\Delta + V$ is an m-sectorial operator with a domain contained in $H^1(\mathbb{H}^n)$. The spectrum of $\Delta + V$ consists of the essential spectrum of $\Delta$ and  isolated eigenvalues of finite algebraic multiplicity.\end{proposition}

To begin with, we review the relevant definitions on unbounded operators and forms in Hilbert spaces. We refer the reader to \cite[Chapter V-VI.]{Kato} for more information. Suppose $\mathbf{H}$ is a Hilbert space with an inner product $(\cdot, \cdot)$.
\begin{itemize}
\item An operator $T$ in $\mathbf{H}$ is said to be accretive if $$\Re (Tu, u) \geq 0 \quad \mbox{for all $u\in \text{Dom}\,(T)$}.$$
\item If $T + \alpha$ for some scalar $\alpha$ is accretive, we say $T$ is quasi-accretive.
\item If an accretive operator $T$ is surjective and also obeys that for any $\lambda$ with $\Re \lambda > 0$  in the resolvent set, $$\|(T + \lambda)^{-1}\| \leq (\Re \lambda)^{-1},$$ $T$ is said to be m-accretive.
\item If $T + \alpha$ for some scalar $\alpha$ is m-accretive, we say $T$ is quasi-m-accretive.
\item We say $T$ is sectorially valued or simply sectorial with vertex $\gamma$ and semi-angle $\theta$, if $$\{\Re(Tu, u) : u \in \text{Dom}(T)\}\subset \{z \in \mathbb{C} : |\arg (z - \gamma)| \leq \theta < \pi/2\}.$$
\item If $T$ is sectorial and quasi-m-accretive, we say $T$ is m-sectorial.

\item We say a quadratic form $\mathfrak{t}$ in $\mathbf{H}$ is sectorially bounded from the left or simply sectorial, if  $$\{\mathfrak{t}[u, u] : \|u\|=1, u \in \text{Dom}(\mathfrak{t})\} \subset \{z \in \mathbb{C} : |\arg(z - \gamma)| \leq \theta < \pi/2, \gamma \in \mathbb{R}\}.$$

\item A sectorial form is said to be closed if that a sequence $\{u_n \in \text{Dom}(\mathfrak{t})\}$ converges to $u$ in $\mathbf{H}$ and  $\mathfrak{t}[u_n - u_m] \rightarrow 0$ as $n, m \rightarrow \infty$ implies that $u\in \text{Dom}(\mathfrak{t})$ and
 $\mathfrak{t}[u_n - u] \rightarrow 0$ as $n \rightarrow \infty$.

\end{itemize}

Let $H_0$ be a self-adjoint, non-negative operator in a Hilbert space $\mathbf{H}$. In addition, we suppose $G_0$ and $G$ are operators from $\mathbf{H}$ to another Hilbert space $\mathbf{G}$ such that  $$\text{Dom} (H_0^{1/2}) \in \text{Dom} (G_0) \cap \text{Dom} (G),$$ but also \begin{equation}\label{eqn : compact operators}\mbox{$G_0(H_0 + 1)^{-1/2}$ and $G(H_0 + 1)^{-1/2}$ are compact.}\end{equation} Under such assumptions, Frank \cite[Lemma B.1, Lemma B.2]{Frank-Simon III} proved that
\begin{lemma}The quadratic form $$\|H_0^{1/2}u\|_{\mathbf{H}} + (Gu, Gu_0)_{\mathbf{G}}$$
 with $\text{Dom}(H_0^{1/2})$ is closed and sectorial. Moreover, it generates an m-sectorial operator $H = H_0 + G^\ast G_0$. The spectrum of $H$ consists of the essential spectrum of $H_0$ and  isolated eigenvalues of finite algebraic multiplicity.\end{lemma}

We want to apply this abstract lemma to the following quadratic form on $L^2(\mathbb{H}^n)$, $$\|\Delta^{1/2} u\|^2_{L^2(\mathbb{H}^n)} + (\sqrt{V}u, \sqrt{|V|}u)_{L^2(\mathbb{H}^n)}.$$ Then $\Delta + V$ would be an m-sectorial operator with a domain contained in $H^1(\mathbb{H}^n)$. It remains to prove \eqref{eqn : compact operators} for $\Delta$, $\sqrt{V}$ and $\sqrt{|V|}$. We have
\begin{lemma}Suppose $V \in L^p(\mathbb{H}^{n})$ with $n/2 \leq p < \infty$ is a complex potential. The operator $\sqrt{|V|} (\Delta + 1)^{-1/2}$ is compact on $L^2(\mathbb{H}^{n + 1})$.\end{lemma}
\begin{proof}
Sobolev's embedding $H^1(\mathbb{H}^n) \subset L^{2n/(n - 2)}(\mathbb{H}^n)$ and the mapping property $(\Delta + 1)^{-1/2}: L^2(\mathbb{H}^n) \rightarrow H^1(\mathbb{H}^n)$ yield that $$(\Delta + 1)^{-1/2}: L^2(\mathbb{H}^n) \longrightarrow L^{2n/(n - 2)}(\mathbb{H}^n).$$ It implies that $$(\Delta + 1)^{-1/2}: L^2(\mathbb{H}^n) \longrightarrow L^{q}(\mathbb{H}^n),\quad \mbox{for any $2 \leq q \leq 2n/(n - 2)$}.$$ From this, one can obtain that for any $W \in L^{2p}(\mathbb{H}^{n})$, \begin{equation}\label{eqn : estimates for compactness}
\|W(\Delta + 1)^{-1/2} f\|_{L^2(\mathbb{H}^n)} \leq C \|W\|_{L^{2p}(\mathbb{H}^n)}\|f\|_{L^{2}(\mathbb{H}^n)}.\end{equation}

We select a sequence $\{W_j \in C_0^\infty(\mathbb{H}^n)\}$ such that $W_j \rightarrow \sqrt{|V|}$ in $L^{2p}(\mathbb{H}^n)$. Then $W_j(\Delta + 1)^{-1/2}$ is an operator which maps $L^2(\mathbb{H}^n)$ to $H^1(\mathbb{H}^n)$. Since $W_j$ is compactly supported, the support of $W_j$ can be thought of as a compact manifold $M_j$ with $C^1$-boundary. The image of $W_j(\Delta + 1)^{-1/2}$ is contained in  $H^1(M_j)$. Since Rellich-Kondrachov theorem implies $H^1(M_j) \subset\subset L^2(M_j) \subset L^2(\mathbb{H}^n)$, we have that $W_j(\Delta + 1)^{-1/2}$ is a compact operator on $L^2(\mathbb{H}^n)$.
On the other hand, \eqref{eqn : estimates for compactness} yields that $$W_j(\Delta + 1)^{-1/2} \longrightarrow \sqrt{|V|}(\Delta + 1)^{-1/2} \quad \mbox{in $L^2(\mathbb{H}^n)$}.$$ Therefore, $\sqrt{|V|}(\Delta + 1)^{-1/2}$ is also a compact operator.

\end{proof}

\end{appendix}

\begin{flushleft}
\vspace{1cm}\textsc{Xi Chen\\
	Shanghai Center for Mathematical Sciences\\
	Fudan University\\
	Shanghai 200438, China\\
	and
	\\Department of Pure Mathematics and Mathematical Statistics\\University of Cambridge\\
	 Cambridge CB3 0WB, UK}

\emph{E-mail address}: \textsf{xi\_chen@fudan.edu.cn} and \textsf{xi.chen@dpmms.cam.ac.uk}

\end{flushleft}

\end{document}